\documentclass[10pt]{article}  

\usepackage{amsmath, amsthm, amsfonts, amssymb}
\usepackage[bookmarks]{hyperref}
\usepackage{indentfirst} 
\usepackage{marvosym}
\usepackage{enumitem}

\usepackage[a4paper]{geometry}

\newtheorem{lemma}{Lemma}[section]
\newtheorem{theorem}[lemma]{Theorem}
\newtheorem{proposition}[lemma]{Proposition}
\newtheorem{corollary}[lemma]{Corollary}
\renewenvironment{proof}[1][\proofname]{{\noindent\bf #1. }}{\qed}
\newtheorem{theoremletters}{Theorem}

\theoremstyle{definition}

\newtheorem{remark}[lemma]{Remark}

\newcommand{\abs}[1]{\ensuremath{|#1|}}
\newcommand{\op}{\operatorname}
\newcommand{\ce}[2]{\pmb{\op{C}}_{#1}(#2)}

\newcommand{\ze}[1]{\pmb{\op{Z}}(#1)}
\newcommand{\fit}[1]{\pmb{\op{F}}(#1)}
\newcommand{\rad}[2]{\pmb{\op{O}}_{#1}(#2)}
\newcommand{\syl}[2]{\op{Syl}_{#1}\hspace*{-0.7mm}\left(#2\right)}
\newcommand{\hall}[2]{\op{Hall}_{#1}\hspace*{-0.7mm}\left(#2\right)}

\usepackage{tikz}

\begin{document}

\title{\bf Groups with triangle-free graphs on $p$-regular classes}

\author{\sc  M.J. Felipe $^{*}$ $\cdot$ M.K. Jean-Philippe $^{\diamond}$ $\cdot$ V. Sotomayor
\thanks{Instituto Universitario de Matemática Pura y Aplicada (IUMPA-UPV), Universitat Polit\`ecnica de Val\`encia, Camino de Vera s/n, 46022 Valencia, Spain. \newline \Letter: \texttt{mfelipe@mat.upv.es, vsotomayor@mat.upv.es}\newline 
ORCID: 0000-0002-6699-3135, 0000-0001-8649-5742  \newline
\indent $^{\diamond}$Departamento de Matemáticas, Instituto Tecnológico de Santo Domingo (INTEC) and Universidad Autónoma de Santo Domingo (UASD), Santo Domingo, Dominican Republic. \newline \Letter: \texttt{1097587@est.intec.edu.do; mjean46@uasd.edu.do}
\newline \rule{6cm}{0.1mm}\newline
The first and last authors are supported by Ayuda a Primeros Proyectos de Investigación (PAID-06-23) from Vice\-rrectorado de Investigación de la Universitat Politècnica de València (UPV), and by Proyecto CIAICO/2021/163 from Generalitat Valenciana (Spain). The results in this paper are part of the second author's Ph.D. thesis.\newline
}}

\date{}

\maketitle

\begin{abstract}
\noindent Let $p$ be a prime. In this paper we classify the $p$-structure of those finite $p$-separable groups such that, given any three non-central conjugacy classes of $p$-regular elements, two of them necessarily have coprime lengths. 

\medskip

\noindent \textbf{Keywords} Finite groups $\cdot$ Conjugacy classes $\cdot$ $p$-regular elements $\cdot$ Common divisor graph $\cdot$ Triangle-free graph

\smallskip

\noindent \textbf{2020 MSC} 20E45 $\cdot$ 20D20 
\end{abstract}


\section{Introduction}

In the following discussion every group is implicitly assumed to be a finite group. Within the theory of finite groups, one of the classical research areas is the study of how the arithmetical properties of the lengths of the conjugacy classes of a group strongly constrain its algebraic structure and vice versa. This intriguing problem has been a target of interest from the nineties up to the present. In this context, an approach that has proven to be very effective to capture arithmetical features of the class lengths of a group $G$ is the \emph{common divisor graph} $\Gamma(G)$. This graph, originally defined by E.A. Bertram, M. Herzog and A. Mann in 1990 (\emph{cf.} \cite{BHM}), is defined as follows: the vertices are the non-central conjugacy classes of $G$, and two classes are adjacent whenever their cardinalities have a common prime divisor. Several results in the literature illustrate that graph-theoretical features of $\Gamma(G)$ are linked to the structure of $G$.

Later on many variations on this topic were introduced by several authors. For instance, A. Beltrán and M.J. Felipe defined in the early 2000s (\emph{cf.} \cite{BF1}) the graph $\Gamma_p(G)$ of a $p$-separable group $G$, for a given prime $p$: the vertices of this graph are the non-central conjugacy classes of $G$ of $p$-regular elements (\emph{i.e.} elements with order not divisible by $p$), and two of them are joined by an edge if and only if their lengths are not coprime. In other words, the graph $\Gamma_p(G)$ is the subgraph of $\Gamma(G)$ induced by those vertices whose representatives have order not divisible by $p$. Several theorems have put forward that certain arithmetical properties of the lengths of these classes are also reflected in the local structure of $G$, and a detailed account on this topic can be found in the expository paper \cite{BFsurvey}.

Both of the above graphs have at most two connected components, and the diameter (in the connected case) is at most three, although these properties cannot be directly derived from the ordinary graph $\Gamma(G)$ to its subgraph $\Gamma_p(G)$ of $p$-regular classes. Indeed, it is worthwhile to mention that, generally, there is no divisibility relation between the class length of an element in a $p$-complement $H$ of $G$ (\emph{i.e. }$H$ is a Hall $p'$-subgroup of $G$) and its corresponding class length in $G$. To illustrate this, take as $G$ the affine semilinear group over the field of eight elements with $p=7$: the class length of an involution in $H$ is $3$, while its class length in $G$ is $7$. Moreover, in this last example $\Gamma(H)$ is non-connected and $\Gamma_p(G)$ is connected, so the structure of both graphs may significantly differ.

Some results for the ordinary graph $\Gamma(G)$ cannot be generalised for the subgraph $\Gamma_p(G)$. For instance, $\Gamma(G)$ is non-connected if and only if $G$ is quasi-Frobenius with abelian kernel and complements (\emph{cf.} \cite{BHM, K}). Recall that $G$ is said to be quasi-Frobenius if $G/\ze{G}$ is Frobenius, and the inverse images in $G$ of the kernel and complements of $G/\ze{G}$ are called the kernel and complements of $G$. It is natural then to ask which local conditions are necessary or sufficient for a $p$-separable group $G$ to have $\Gamma_p(G)$ non-connected. For instance, inspired by the ordinary case, is $\Gamma_p(G)$ non-connected when a $p$-complement of $G$ is quasi-Frobenius with abelian kernel and complements? This is certainly false, and a counterexample is given by the affine semilinear group described in the above paragraph. In fact this question looks more articulate, and a wider analysis was taken further by Beltrán and Felipe in \cite{BF3}. Nevertheless, a complete understanding of the $p$-structure of groups $G$ with $\Gamma_p(G)$ non-connected has not been reached, since there are still some unsolved problems (see Remark \ref{remark_intro} below).

This paper is a contribution to a better understanding of how the local structure of a $p$-separable group $G$ is influenced by the graph-theoretical properties of $\Gamma_p(G)$, for a given prime $p$. More concretely, our goal is to classify the structure of the $p$-complements of $G$ when $\Gamma_p(G)$ contains no triangles. This research is motivated by a result due to M. Fang and P. Zhang (\emph{cf.} \cite{FZ}), who completely obtained the groups $G$ whose ordinary graph $\Gamma(G)$ has no triangles: the symmetric group of degree $3$, the dihedral group of order $10$, the three pairwise non-isomorphic non-abelian groups of order $12$, and the non-abelian group of order $21$. Our main result is the following.

\begin{theoremletters}
\label{teoA}
Let $G$ be a $p$-separable group, for a given prime $p$, and let $H$ be a non-central $p$-complement of $G$. If $\Gamma_p(G)$ has no triangles, then one of the following cases holds:
\vspace*{-2mm}
\begin{enumerate}
\setlength{\itemsep}{-1mm}
\item[\emph{i)}] $H$ is a $q$-group, for some prime $q\neq p$.
\item[\emph{ii)}] $H$ is a quasi-Frobenius $\{q,r\}$-group with abelian kernel and complements for certain primes $q$ and $r$ both distinct from $p$, and $\ze{H}=H\cap\ze{G}$ is either trivial or cyclic of order two.
\item[\emph{iii)}] $H$ is isomorphic to $(C_5\times C_5)\rtimes Q_8$.
\end{enumerate}
\vspace*{-2mm}
\noindent In particular $G$ is soluble, and $\Gamma_p(G)$ coincides with one of the next six graphs.

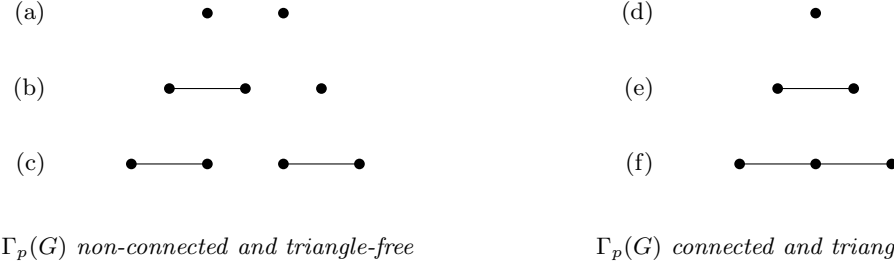
\begin{center}
\begin{tikzpicture}
  \coordinate (A) at (1,0);
  \coordinate (B) at (2,0);
  \coordinate (C) at (0.5,-1);
  \coordinate (D) at (1.5,-1);
  \coordinate (E) at (2.5,-1);
  \coordinate (F) at (0,-2);
  \coordinate (G) at (1,-2);
  \coordinate (H) at (2,-2);
  \coordinate (I) at (3,-2);
  \coordinate (J) at (9,0);
  \coordinate (K) at (8.5,-1);
  \coordinate (L) at (9.5,-1);
  \coordinate (M) at (8,-2);
  \coordinate (N) at (9,-2);
  \coordinate (O) at (10,-2);

  \fill (A) circle (2pt);
  \fill (B) circle (2pt);
  \fill (C) circle (2pt);
  \fill (D) circle (2pt);
  \fill (E) circle (2pt);
  \fill (F) circle (2pt);
  \fill (G) circle (2pt);
  \fill (H) circle (2pt);
  \fill (I) circle (2pt);
  \fill (J) circle (2pt);
  \fill (K) circle (2pt);
  \fill (L) circle (2pt);
  \fill (M) circle (2pt);
  \fill (N) circle (2pt);
  \fill (O) circle (2pt);
  
  \draw (D) -- (C);
  \draw (F) -- (G);
  \draw (H) -- (I);
  \draw (K) -- (L);
  \draw (M) -- (N) -- (O);
  
  \node[above] at (1,-3.4) {{\small $\Gamma_p(G)$ non-connected and triangle-free}};
  \node[above] at (8.5,-3.4) {{\small $\Gamma_p(G)$ connected and triangle-free}};
  \node[left] at (-1,0) {{\small \emph{(a)}}};
  \node[left] at (-1,-1) {{\small \emph{(b)}}};
  \node[left] at (-1,-2) {{\small \emph{(c)}}};
    \node[left] at (7,0) {{\small \emph{(d)}}};
  \node[left] at (7,-1) {{\small \emph{(e)}}};
  \node[left] at (7,-2) {{\small \emph{(f)}}};
\end{tikzpicture}
\end{center}
Case \emph{i)} corresponds to \emph{(d)} and \emph{(e)}, case \emph{ii)} corresponds to \emph{(a)}, \emph{(b)}, \emph{(c)} and \emph{(e)}, and case \emph{iii)} corresponds to \emph{(f)}.
\end{theoremletters}

It is worth mentioning that none of the graphs (d), (e) and (f) of Theorem \ref{teoA} can occur as graphs $\Gamma(G)$ for a group $G$ (\emph{cf.} \cite{FZ}). 

The details concerning the structure of the $p$-complements of $G$ in each case are described in Sections \ref{sec-non-connected} and \ref{sec-connected}. In particular, in Corollary \ref{cor_disconnected_2} we make further progress on the aforementioned open problem concerning the $p$-structure of $G$ when $\Gamma_p(G)$ is non-connected.

It will be sometimes possible to deduce that $G$ has either two, three or four $p$-regular classes, and these groups were characterised in \cite{N1, N2, T} when $\rad{p}{G}=1$. Since the number of $p$-regular classes of $G$ and $G/\rad{p}{G}$ is the same, then we will occasionally be able to obtain the structure of $G/\rad{p}{G}$. However, we remark that central $p$-regular classes are also studied in those papers, whilst in our analysis the number of those classes is a priori unknown. Besides, some of the groups that we consider may have five or more $p$-regular classes (see Remark \ref{remark_many_classes}), and such groups have not been characterised yet.

Finally, we illustrate with examples throughout the paper the extent to which these results are best possible.


\section{Preliminaries}

In the sequel, if $n$ is a positive integer, then we write $\pi(n)$ for the set of prime divisors of $n$. In addition, $\pi(G)$ is the set of prime divisors of $\abs{G}$ and, if $B=b^G$ is a conjugacy class of $G$, then $\pi(B)$ is the set of prime divisors of $|B|=|G:\ce{G}{b}|$. As usual, given a prime $p$, the set of all Sylow $p$-subgroups of $G$ is denoted by $\syl{p}{G}$, and $\hall{\pi}{G}$ is the set of all Hall $\pi$-subgroups of $G$ for a set of primes $\pi$. We write $\pi'$ for the set of primes that do not belong to $\pi$. The order of an element $g\in G$ is $o(g)$, and for any subset $\sigma\subseteq \pi(o(g))$, the $\sigma$-part of $g$ is denoted by $g_{\sigma}$. The largest $\pi$-number that divides the positive integer $n$ is $n_{\pi}$. The number of $p$-regular conjugacy classes of $G$ is denoted by $k_{p'}(G)$. Finally, cyclic, dihedral, semidihedral and elementary abelian groups of order $n$ are denoted by $C_n$, $D_n$, $SD_n$ and $E_n$, respectively; symmetric and alternating groups of degree $n$ are written $\Sigma_n$ and $A_n$, respectively, and $Q_8$ is the quaternion group of order $8$. The remaining notation and terminology used is standard in the frameworks of group theory and graph theory.

The following well-known lemma is used frequently and without further reference.

\begin{lemma}
Let $N$ be a normal subgroup of a group $G$. Then:
\vspace*{-2mm}
\begin{itemize}
\setlength{\itemsep}{-1mm}
\item[\emph{(a)}] $|x^N|$ divides $|x^G|$, for any $x\in N$.
\item[\emph{(b)}] $|(xN)^{G/N}|$ divides $|x^G|$, for any $x\in G$.
\item[\emph{(c)}] If $xN\in G/N$ is a $p$-regular element, then $xN=yN$ for some $p$-regular element $y\in G$.
\item[\emph{(d)}] If $x,y\in G$ have coprime orders and they commute, then $\ce{G}{xy}=\ce{G}{x}\cap\ce{G}{y}$. In particular, both $|x^G|$ and $|y^G|$ divide $|(xy)^G|$.
\end{itemize}
\end{lemma}

In certain situations of our study, we will be able to provide not only the structure of the $p$-complements of $G$, but also the structure of $G/\rad{p}{G}$, due to the fact below.

\begin{lemma}
\label{kpG}
Let $G$ be a group, and let $p$ be a prime. Then $k_{p'}(G)=k_{p'}(G/\rad{p}{G})$.
\end{lemma}

\begin{proof}
Certainly $k_{p'}(G/\rad{p}{G})\leq k_{p'}(G)$, so it is enough to prove the other inequality. By contradiction, let us suppose that $a^G\neq b^G$ are two $p$-regular classes of $G$ such that $(a\,\rad{p}{G})^{G/\rad{p}{G}}=(b\,\rad{p}{G})^{G/\rad{p}{G}}$. Then for some $g\in G$ it holds $a^g\in b\,\rad{p}{G}\subseteq \langle b\rangle \rad{p}{G}$. Thus $a^g=zb$ for some $z\in\rad{p}{G}$, and also there exists certain $x\in \rad{p}{G}$ such that $(a^g)^x=a^{gx}\in\langle b\rangle$. It follows $b(a^{gx})^{-1}z^x=b(b^{-1})^x\in [\langle b\rangle, \rad{p}{G}]$, and since $z^x\in\rad{p}{G}$, we deduce $b(a^{gx})^{-1}\in\rad{p}{G}$. Consequently $b(a^{gx})^{-1}\in \rad{p}{G}\cap \langle b\rangle=1$ and $b=a^{gx}$, a contradiction.
\end{proof}

The situation when $\Gamma_p(G)$, the common divisor graph for $p$-regular conjugacy classes, is non-connected was mainly studied in \cite{BF3}.

\begin{theorem}
\label{disconnected-p-reg}
Let $G$ be a $p$-separable group, for a given prime $p$, and let $H$ be a $p$-complement of $G$. Suppose that $\Gamma_p(G)$ is non-connected. Let $B_0$ be a non-central $p$-regular class of maximal cardinality, and let $\pi_0$ be the set of prime divisors of the lengths of all conjugacy classes lying in the same connected component as $B_0$.
\begin{itemize}
\setlength{\itemsep}{-1mm}
\item[\emph{(a)}] If $p\notin \pi_0$, then $G$ is $p$-nilpotent, $H$ is quasi-Frobenius with abelian kernel and complements, and $\ze{H}=H\cap\ze{G}$. Moreover, each complement of $H$ is centralised by some Sylow $p$-subgroup of $G$.
\item[\emph{(b)}] If $p\in\pi_0$, and either $|\pi_0|\geq 3$ or $G$ has abelian Hall $\pi_0\smallsetminus\{p\}$-subgroups, then $\ze{H}=H\cap\ze{G}$, and $H$ is a quasi-Frobenius group with abelian kernel and complements.
\end{itemize}
\end{theorem}

\begin{proof}
Statement (a) follows from \cite[Theorem 5 (a)]{BF1} and \cite[Theorem 8]{BF3}, and statement (b) can be obtained from \cite[Corollary 10 (b), Theorem 12 (b)]{BF3}.
\end{proof}

\medskip

\begin{remark}
\label{remark_intro}
Following the notation of the previous theorem, we point out that by \cite[Theorem 4 (b)]{BF3}, if $p\in\pi_0$, then $|\pi_0|\geq 2$. Hence, as mentioned in the Introduction, it remains unsolved in the above theorem the situation where $\pi_0=\{p,q\}$ and a Sylow $q$-subgroup of $G$ is non-abelian.
\end{remark}

Next we collect some features of non-central $p$-regular classes of maximal length.

\begin{proposition}
\label{prop}
Let $p$ be a prime, and suppose that $G$ is a $p$-separable group. Let $B_0$ be a non-central $p$-regular class of maximal length. Set $$ S=\langle D \; : \; D \text{ is $p$-regular, non-central, and} \; (|B_0|,|D|)=1\rangle. $$ Then $S$ is an abelian normal $p'$-subgroup of $G$, and if $\ze{G}_{p'}$ denotes the $p$-complement of $\ze{G}$, then $\ze{G}_{p'}\leqslant S$ and $\pi(S/\ze{G}_{p'}) \subseteq \pi(B_0)$.
\end{proposition}

\begin{proof}
This is proved within \cite[Proposition 1]{BF1} when $\Gamma_p(G)$ is connected, and in \cite[Theorem 4]{BF3} when it is non-connected.
\end{proof}

\medskip

The following technical lemma will be very useful along the paper.

\begin{lemma}
\label{lemma_centre}
Let $G$ be a $p$-separable group for a fixed prime $p$, and let $H$ be a $p$-complement of $G$. If the vertex set of $\Gamma_p(G)$ is non-empty, there is no triangle in $\Gamma_p(G)$, and $H$ is not of prime power order, then $|H\cap\ze{G}|\leq 2$.
\end{lemma}

\begin{proof}
Let us suppose that $H\cap \ze{G}>1$. Since $H$ has not prime power order, and there are vertices in $\Gamma_p(G)$, then we can then take a $q$-element $x\in H\smallsetminus (H\cap\ze{G})$ and a $r$-element $z\in H\cap \ze{G}$ for two different primes $q\neq r$. If $r>2$, then it is easy to see by the orders of the elements that $x^G$, $(zx)^G$ and $(z^2x)^G$ are different $p$-regular classes that form a triangle in $\Gamma_p(G)$, which is not possible. Hence $r=2$. Now if $|H\cap\ze{G}|_2>2$, then there exist two different non-trivial $2$-elements $z_1,z_2\in H\cap \ze{G}$, and similarly as before $x^G$, $(z_1x)^G$ and $(z_2x)^G$ form a triangle in $\Gamma_p(G)$, a contradiction. Consequently, $|H\cap\ze{G}|_2=2$. 

Finally, we claim that $|H\cap\ze{G}|_q=1$. Arguing by contradiction, let us suppose that there exists $1\neq y\in H\cap\ze{G}$ a $q$-element. We first show that there is only one non-central class of $q$-elements, which is $x^G$. Let $z\in H\cap \ze{G}$ be the involution considered above. Observe that $(xy)^G=x^Gy$ is necessarily equal to $x^G$, since otherwise we get a triangle in $\Gamma_p(G)$ with $x^G$, $(xz)^G$ and $(xy)^G$, which is not possible. Hence $x^G=x^G\langle y \rangle$, which implies that $|\langle y\rangle|$ divides $|x^G|$. So $q$ divides $|x^G|$, for every $q$-element $x\in H\smallsetminus (H\cap \ze{G})$. Therefore, by the triangle-free assumption, and the fact that $|x^G|=|x^Gz|$, we deduce that $x^G$ is the unique non-central class of $q$-elements. Observe that $|x^G|$ is odd: on the one hand this is clear when $2\notin \pi(H/(H\cap \ze{G}))$, and on the other hand if there is a $2$-element in $H\smallsetminus (H\cap\ze{G})$, then using similar arguments one can show that $2$ must divide its class size in $G$, so $|x^G|$ cannot be divisible by $2$. Since $x^G$ is real, then $x$ is an involution, a contradiction.
\end{proof}

The theorem below is a relevant result due to Higman concerning soluble groups whose elements all have prime power order.

\begin{theorem}[\text{\cite[Theorem 1]{H}}]
\label{CP}
Let $G$ be a soluble group such that every element has prime power order. Let $t$ be the prime such that $G$ has a non-trivial normal $t$-subgroup $M$. Then $G=M$ or $G/M$ is either:
\vspace*{-2mm}
\begin{enumerate}
\setlength{\itemsep}{-1mm}
	\item[\emph{(a)}] a cyclic group whose order is a power of a prime other than $t$; or
	\item[\emph{(b)}] a generalized quaternion group, $t$ being odd; or
	\item[\emph{(c)}] a group of order $t^as^b$ with cyclic Sylow subgroups, $s$ being a prime of the form $kt^a+1$, $k\in\mathbb{N}$.
\end{enumerate}
\vspace*{-2mm}
Thus, $G$ has order divisible by at most two primes, and $G/M$ is metabelian.
\end{theorem}

It is immediate to see that $G$ is a Frobenius group in cases (a) and (b) above, while $G/M$ is a Frobenius group in case (c).

We close this preliminary section with the next arithmetical result.

\begin{proposition}
\label{clave}
Let $G$ be a $\pi$-separable group, for a set of primes $\pi$.
\vspace*{-2mm}
\begin{enumerate}
\setlength{\itemsep}{-1mm}
	\item[\emph{(a)}] Each $\pi$-element has class size a $\pi$-number if and only if $G=\rad{\pi}{G}\times\rad{\pi'}{G}$.
	
	\item[\emph{(b)}] Each $\pi$-element has class size a $\pi'$-number if and only if $G$ has abelian Hall $\pi$-subgroups.
\end{enumerate}
\end{proposition}

\begin{proof}
The first assertion can be easily deduced from \cite[Lemma 9]{FMS}, and the second one is a consequence of \cite[Proposition 1]{FMS} when the trivial factorisation is considered.
\end{proof}


\section{The non-connected case}
\label{sec-non-connected}

In this section we focus on $p$-separable groups $G$ such that $\Gamma_p(G)$ is non-connected and triangle-free. Note that $\Gamma_p(G)$ has two complete connected components by \cite[Theorems 1 and 3]{BF1}, so we have only three situations to discuss: $\Gamma_p(G)$ has either two vertices and no edges, three vertices and one edge, or four vertices and two edges. Let us first study the general structure of the $p$-complements of $G$ in these three situations.

\begin{proposition}
\label{cor_disconnected}
Let $G$ be a $p$-separable group for a fixed prime $p$, and let $H$ be a $p$-complement of $G$. If $\Gamma_p(G)$ is non-connected and without triangles, then $H$ is a quasi-Frobenius group with abelian kernel and complements, and $\ze{H}=H\cap\ze{G}$.
\end{proposition}

\begin{proof}
Let us first suppose that a $p$-regular class of maximal length is not an isolated vertex of $\Gamma_p(G)$, so it lies in a connected component of this graph, say $\{a^G,b^G\}$ with $a^G\neq b^G$. Hence in the other component of $\Gamma_p(G)$ there are at most two $p$-regular classes, say $\{c^G, d^G\}$ (possibly $c^G=d^G$). Set $\pi_a=\pi(a^G)$, $\pi_b=\pi(b^G)$, $\pi_c=\pi(c^G)$ and $\pi_d=\pi(d^G)$. We deduce by Theorem \ref{disconnected-p-reg} that $\pi_a=\pi_b=\{p,q\}$, and it is enough to prove that a Sylow $q$-subgroup of $G$ is abelian by statement (b) of that result. Note that $S=\langle c^G, d^G\rangle$ is an abelian normal $q$-subgroup of $G$ by Proposition \ref{prop}. Let $a=a_ra_{r'}$ for some prime $r$ such that $a_r\notin\ze{G}$. Since $1\neq |a_r^G|$ divides $|a^G|$, then either $a_r^G=a^G$ or $a_r^G=b^G$, and so either $a$ or $b$ has prime power order. We may thus assume that $a$ has prime power order. If $b$ also has prime power order, and moreover both are $q$-elements, as $q\notin \pi_c\cup\pi_d$, then there exist conjugates of $a$ that commute with $c$ and $d$, and there also exist conjugates of $b$ that commute with $c$ and $d$, respectively. Clearly $G=\ce{G}{a}\ce{G}{c}=\ce{G}{a}\ce{G}{d}$ because the corresponding class lengths are coprime, so it follows $ac=ca$ and $ad=da$, and analogously $bc=cb$ and $bd=db$. In fact $a,b\in\ce{G}{c^G}\cap \ce{G}{d^G}\subseteq \ce{G}{S}\unlhd G$. Consequently $a^G,b^G\subseteq \ce{G}{S}$, so $S\leqslant\ze{H}$ which leads to the contradiction $|c^G|=|d^G|=1$. If $a$ and $b$ have orders $r^a$ and $s^b$ for two primes $r\neq s$, then one of them is different from $q$ necessarily, so there exists some $g\in G$ with $a^gb=ba^g$. It follows that $|(a^gb)^G|$ is divisible by both $|a^G|$ and $|b^G|$, and then the $p$-regular class $(a^gb)^G$ cannot be either $c^G$ or $d^G$; but its representative does not have prime power order, so this class cannot be either $a^G$ or $b^G$, a contradiction. Hence both elements are $r$-elements with $r\neq q$, and so for each non-central $x\in Q\in\syl{q}{G}$ it follows $x^G=c^G$ or $x^G=d^G$, so $x\in S$. As a consequence $Q=S(Q\cap \ze{G})$, and by Proposition \ref{prop} we get $Q\cap \ze{G}\leqslant\ze{G}_{p'}\leqslant S$, which yields that $Q=S$ is abelian, as wanted. Now we may suppose that $b$ has not prime power order, so, at least, $b_r, b_s\notin \ze{G}$ for two different primes $r\neq s$. One of these primes needs to be coprime with $\pi(o(a))$, so for instance we may affirm that $r\notin\pi(o(a))$. But then $b_r^G$ should be either $a^G$ or $b^G$, which is not possible by orders.

We may now suppose that a $p$-regular class $b^G$ of maximal length is an isolated vertex of $\Gamma_p(G)$, so in the other connected component there are at most two $p$-regular classes, say $\{a^G, c^G\}$ (possibly $a^G=c^G$). Set $\pi_a=\pi(a^G)$, $\pi_b=\pi(b^G)$ and $\pi_c=\pi(c^G)$. Since $b^G$ is an isolated vertex, certainly we may assume that $b$ is a prime power order element. By Theorem \ref{disconnected-p-reg}, we may suppose that $\pi_b=\{p,q\}$ for certain prime $q\neq p$, and again it is enough to prove that a Sylow $q$-subgroup of $G$ is abelian. Applying Proposition \ref{prop}, the $q$-subgroup $S=\langle a^G, c^G\rangle$ is abelian and normal in $G$. If $b$ is a $q$-element, as $q\notin \pi_a\cup\pi_c$, then we can argue as above to deduce $ab=ba$ and $bc=cb$. In fact $b^G\subset\ce{G}{S}$, so $S\leqslant \ze{H}$ and $|a^G|=|c^G|=1$, which is not possible. We deduce that $b$ is a $q$-regular element, and so for each non-central $x\in Q\in\syl{q}{G}$ it follows $x^G=a^G$ or $x^G=c^G$, so $x\in S$. As a consequence $Q=S(Q\cap \ze{G})=S$ is abelian, as desired.
\end{proof}

\medskip

Notice that, in view of Theorem \ref{disconnected-p-reg} (b), a problem that remains now open is whether the $p$-complements of $G$ are quasi-Frobenius groups (with abelian kernel and complements) when $\Gamma_p(G)$ is non-connected, with triangles, $\pi_0=\{p,q\}$, and the Sylow $q$-subgroups of $G$ are non-abelian.

Next we further investigate the simplest situation where $\Gamma_p(G)$ is non-connected and triangle-free.

\begin{proposition}
\label{two-disconnected}
Let $G$ be a $p$-separable group for a fixed prime $p$. If $\Gamma_p(G)$ has two vertices and no edges, then one of the following statements holds:
\vspace*{-2mm}
\begin{enumerate}
\setlength{\itemsep}{-1mm}
	\item[\emph{(a)}] $G/\rad{p}{G}\cong \Sigma_3$, where $p\notin \{2,3\}$.
	\item[\emph{(b)}] $G/\rad{p}{G}\cong C_r\rtimes (C_2\times C_{p^n})$, where $p\neq 2$ and $r=2p^n+1$ is a prime.
	\item[\emph{(c)}] $G/\rad{p}{G}\cong E_{3^l}\rtimes (C_2\times C_{p^n})$, where $p\neq 2$ and $3^l=2p^n+1$.
\end{enumerate}
\vspace*{-2mm}
In particular, every $p$-complement of $G$ is a Frobenius group with $q$-elementary abelian kernel and complements of order $2$, for some odd prime $q\neq p$.
\end{proposition}

\begin{proof}
Suppose that the non-central $p$-regular classes are $\{a^G, b^G\}$. We first claim that $H\cap \ze{G}=1$. By contradiction, let us suppose that there exists $1\neq z\in H\cap \ze{G}$. Then $(az)^G=a^Gz$ and so $a^Gz=a^G$ necessarily, which yields that $a^G\langle z\rangle = a^G$. It follows that $|\langle z\rangle|$ divides $|a^G|$, and similarly we can deduce that $|\langle z\rangle|$ divides $|b^G|$. Since $a^G$ and $b^G$ have coprime lengths, then $z=1$, as desired. In particular $k_{p'}(G)=3$, and $k_{p'}(G/\rad{p}{G})=3$ by Lemma \ref{kpG}. In order to check in a easier way the classification provided in \cite{N1,N2}, we use that a $p$-complement $H$ of $G$ is a Frobenius group with abelian kernel and complements due to Proposition \ref{cor_disconnected}, and so $\pi(H)=2$ due to the structure of $\Gamma_p(G)$. Furthermore, observe that $p$ is odd, since both $p$-regular conjugacy classes are real (\emph{i.e.} $a^G=(a^{-1})^G$ and $b^G=(b^{-1})^G$), and one of them necessarily has odd length, say $a^G$, so we get that $o(a)=2\in\pi(H)$. Hence the unique possibility from \cite[Theorem B]{N1} is (2), which leads to our statement (a); whilst the unique possibilities from \cite{N2} are (1) and (2), that yield (b) and (c).
\end{proof}

\medskip

\begin{remark}
It is easy to check that all cases in Proposition \ref{two-disconnected} actually occur. It is also worth mentioning that the converse does not hold. For instance, take $G=E_{25} \rtimes \Sigma_3$ where $E_{25}=\langle x \rangle \times \langle y\rangle$, $x^{(1,2)}=x$, $y^{(1,2)}=x^4y$, and $\langle (1,2,3)\rangle$ permutes the non-trivial $5$-elements transitively. If $p=5$, then $G/\rad{5}{G}\cong \Sigma_3$, but the non-central $p$-regular classes of $G$ are $\{(1,2)^G, (1,2,3)^G\}$, and their lengths are both divisible by $5$.
\end{remark}

Let us address the remaining cases where $\Gamma_p(G)$ is non-connected and without triangles.

\begin{proposition}
\label{three-disconnected}
Let $G$ be a $p$-separable group for a fixed prime $p$, and let $H$ be a $p$-complement of $G$. If $\Gamma_p(G)$ has three vertices and one edge, then $H$ is a Frobenius $\{q,r\}$-group with abelian kernel and complements, for two different primes $q$ and $r$ both distinct from $p$.

\noindent Moreover, if $p$ is odd, then one of the following statements holds:
\vspace*{-2mm}
\begin{enumerate}
\setlength{\itemsep}{-1mm}
	\item[\emph{(a)}] $G/\rad{p}{G}\cong A_4$, where $p\notin \{2,3\}$.
	\item[\emph{(b)}] $G/\rad{p}{G}\cong D_{10}$, where $p\notin \{2,5\}$.
	\item[\emph{(c)}] $G/\rad{p}{G}\cong E_{q^t} \rtimes (C_2 \times C_{p^s})$, where $q^t-1=4p^s$.
	\item[\emph{(d)}] $G/\rad{p}{G}\cong E_{16} \rtimes C_{15}$, where $p=5$ and $C_{15}$ permutes the involutions transitively.
\end{enumerate}
\end{proposition}

\begin{proof}
We know by Proposition \ref{cor_disconnected} that $\ze{H}=H\cap\ze{G}$ and that $H$ is a quasi-Frobenius group with abelian kernel and complements. In particular $H$ is soluble. In order to prove the first claim, it remains to show that $H\cap\ze{G}=1$; note that the assertion $|\pi(H)|= 2$ will then directly follows, since otherwise the structure of $\Gamma_p(G)$ yields that every element of $H$ has prime power order, and by Proposition \ref{CP} we get the contradiction $|\pi(H)|\leq 2$.

Let us suppose $H\cap\ze{G}>1$. By Lemma \ref{lemma_centre} it holds $|H\cap\ze{G}|=2$, so we can take an involution $z\in H\cap\ze{G}$. Let $a^G$ be the non-central $p$-regular class that is non-adjacent in $\Gamma_p(G)$ to the other two vertices, say $\{b^G, c^G\}$. Certainly $(az)^G=a^Gz=a^G$, so $\langle z\rangle a^G=a^G$ and $|\langle z\rangle|=2$ divides $|a^G|$. In addition $o(a)=o(az)=\op{mcm}(o(a),2)$, and since we may suppose that $a$ has prime power order, then $a$ is necessarily a $2$-element. Now $2\notin \pi(b^G)\cup \pi(c^G)$, so $ab=ba$ and $ac=ca$ because they have coprime class sizes. If either $b_{2'}\notin\ze{G}$ or $c_{2'}\notin\ze{G}$, then we get a contradiction since we could built in $\Gamma_p(G)$ a conjugacy class that would be adjacent to both connected components. Therefore $\pi(H/\ze{H})=\{2\}$, and this certainly cannot happen.

The second part of the proof is an immediate application of the main result of \cite{T}, because $G$, and so $G/\rad{p}{G})$, has four $p$-regular classes, and $|\pi(H)|=2$. Note that the unique cases of \cite[Main Theorem]{T} where a $p$-complement of $G/\rad{p}{G}$ has not prime power order are B) 1 (c) and (d), and B) 4, 5 and 6. But B) 4 cannot actually occur because $\Gamma_p(G/\rad{p}{G})$ is a path of length two, and by class length divisibility it would follow that $\Gamma_p(G)$ is connected.
\end{proof}

\medskip

\begin{remark}
It is easy to see that all cases in Proposition \ref{three-disconnected} occur. In addition, $p$ may be even: for instance, take $p=2$ and $G=C_7\rtimes C_6$ where $C_6$ permutes the non-trivial elements of order $7$ transitively, so the three non-central $p$-regular classes have lengths $6$, $7$ and $7$. Nonetheless, we recall that groups with four $2$-regular classes have not yet been classified.
\end{remark}

\medskip

\begin{proposition}
\label{four-disconnected}
Let $G$ be a $p$-separable group for a fixed prime $p$, and let $H$ be a $p$-complement of $G$.  If $\Gamma_p(G)$ has two connected components, each one formed by two adjacent vertices, then $\ze{H}=H\cap\ze{G}$ and $H$ is a quasi-Frobenius $\{q,r\}$-group with abelian kernel and complements, for two different primes $q$ and $r$ both distinct from $p$. Moreover, $\ze{H}$ is either trivial or cyclic of order two.
\end{proposition}

\begin{proof}
By Proposition \ref{cor_disconnected} it holds that a $p$-complement $H$ of $G$ satisfies $\ze{H}=H\cap\ze{G}$ and that $H$ is a quasi-Frobenius group with abelian kernel and complements. By Lemma \ref{lemma_centre}, it holds $|H\cap\ze{G}|\leq 2$. 

Next we claim $|\pi(H)|=2$. Suppose that there exists $x(H\cap\ze{G})\in H/(H\cap \ze{G})$ of order divisible by exactly two primes $q\neq r$. Then $x=x_qx_r$ and clearly $x_q,x_r\notin\ze{G}$. Since $|x_q^G|$ and $|x_r^G|$ divide $|x^G|$, and these classes are distinct by orders, then we obtain in $\Gamma_p(G)$ the following path  $$x_q^G \mbox{ --- } x^G \mbox{ --- } x_r^G,$$ but this contradicts the assumptions. Hence all elements of $H/(H\cap\ze{G})$ have prime power order, so $|\pi(H/(H\cap\ze{G})|=2$. If $H\cap\ze{G}=1$, then there is nothing to prove. Suppose $|H\cap\ze{G}|=2$ and that $|H/H\cap\ze{G}|$ is divisible by two odd primes $q\neq r$. Take $x,y\in H\smallsetminus (H\cap \ze{G})$ a $q$-element and a $r$-element, respectively, and take the involution $z\in H\cap\ze{G}$. It is clear that the four vertices of $\Gamma_p(G)$ are $x^G$, $(xz)^G$, $y^G$ and $(yz)^G$, and thus the set of vertices of $\Gamma_p(G/\ze{G})$ is a non-trivial subset of $\{(x\ze{G})^{G/\ze{G}}, (y\ze{G})^{G/\ze{G}}\}$. If these are indeed the vertices of $\Gamma_p(G/\ze{G})$, since they must be non-adjacent, then by Proposition \ref{two-disconnected} we get the contradiction $2\in\pi(H/H\cap\ze{G})$. So we may suppose that $x\ze{G}\in\ze{G/\ze{G}}$, and so $x^hx^{-1}\in \ze{G}\cap H=\ze{H}$ for all $h\in H$. As $|\ze{H}|=2$, then $2=|x^H|=|H:\ce{H}{x}|$ divides $|H:\ze{H}|=|H/(H\cap\ze{G})|$. This final contradiction finishes the proof.
\end{proof}

\medskip

\begin{remark}
In the last result $k_{p'}(G)\in \{5,6\}$, but the $p$-structure of this kind of groups has not yet been analysed. An example where $H\cap\ze{G}$ is non-trivial, which differs from Propositions \ref{two-disconnected} and \ref{three-disconnected}, is the following one: take $G=C_3\rtimes C_4$ and $p\notin\{2,3\}$, so the non-central $p$-regular classes of $G$ have lengths $2,2,3$ and $3$, but $H=G$ and $\ze{G}$ has order $2$.
\end{remark}

\smallskip

\begin{corollary}
\label{cor_disconnected_2}
Let $G$ be a $p$-separable group for a fixed prime $p$, and let $H$ be a $p$-complement of $G$. If $\Gamma_p(G)$ is non-connected and without triangles, then $H$ is a quasi-Frobenius $\{q,r\}$-group with abelian kernel and complements for certain primes $q$ and $r$ both distinct from $p$, and $\ze{H}=H\cap\ze{G}$ is either trivial or cyclic of order two.
\end{corollary}

\begin{proof}
In view of Proposition \ref{cor_disconnected}, it remains to show $|\pi(H)|=2$ and $|\ze{H}|\leq 2$. Recall that each connected component of $\Gamma_p(G)$ must be complete by \cite[Theorem 3]{BF1}, so we only have the three situations discussed in Propositions \ref{two-disconnected}, \ref{three-disconnected} and \ref{four-disconnected}. A simple glance to these three results gives the desired conclusions. 
\end{proof}

\section{The connected case}
\label{sec-connected}

The simplest case where $\Gamma_p(G)$ is connected and without triangles is when it consists of a single vertex. We remark that this situation for the ordinary graph $\Gamma(G)$ cannot happen, since any group is generated by a set of representatives of its conjugacy classes.

\begin{proposition}
\label{unique_vertex}
Let $G$ be a group, and let $p$ be a prime divisor of $|G|$.  If $\Gamma_p(G)$ only consists in one vertex $a^G$, where $a\in G$ is a $p$-regular element, then $G$ is a $\{2,q\}$-group for some odd prime $q$, so $p\in\{2,q\}$, and any $p$-complement $H$ of $G$ satisfies that $H/(H\cap \ze{G})$ is elementary abelian. 

\noindent Moreover, $H$ is abelian if and only if one of the following statements holds:
\vspace*{-2mm}
\begin{enumerate}
\setlength{\itemsep}{-1mm}
	\item[\emph{(a)}] $G/\rad{p}{G}\cong C_2$ with $p$ odd.
	\item[\emph{(b)}] $G/\rad{p}{G}\cong E_{3^2}\rtimes C_8$ with $p=2$.
	\item[\emph{(c)}] $G/\rad{p}{G}\cong E_{3^2}\rtimes Q_8$ with $p=2$.
	\item[\emph{(d)}] $G/\rad{p}{G}\cong E_{3^2}\rtimes SD_{16}$ with $p=2$.
	\item[\emph{(e)}] $G/\rad{p}{G}\cong C_q \rtimes C_{2^n}$ with $p=2$ and $q=2^n+1$ a Fermat prime.
	\item[\emph{(f)}] $G/\rad{p}{G}\cong C_{2^n} \rtimes C_p$ with $p=2^n-1$ a Mersenne prime.
\end{enumerate}
\end{proposition}

\begin{proof}
From the decomposition of $a$ as product of pairwise commuting elements of prime power order, and from the structure of $\Gamma_p(G)$, we may assume that $a$ is a $r$-element for some prime $r\neq p$. Let $w\in \ze{G}$ be a $p'$-element. Then certainly $wa\notin\ze{G}$, so $wa$ is conjugated to $a$, which implies that $w$ is a $r$-element too. Thus $\ze{G}$ is a $\{p,r\}$-group, and in particular so is $G/\ze{G}$ because $a^G$ is the unique non-central $p$-regular conjugacy class. It follows that $|\pi(G)|=2$, and thus it is soluble by Burnside's theorem. Further, if $G$ has odd order, then $a^G$ is a real conjugacy class with odd length, so $o(a)=2$ which is a contradiction. Therefore $|G|$ is even.

Let $H$ be a $p$-complement of $G$. We claim that $\overline{H}=H/(H\cap\ze{G})$ is elementary abelian. Note that $H=(H\cap \ze{G})\,\dot{\cup}\, a_1^H\,\dot{\cup}\,\cdots\,\dot{\cup} \,a_n^H$ for some integer $n\geq 1$, where $a_1^H,\ldots, a_n^H\subseteq a^G$ are conjugacy classes of $H$ that are conjugated in $G$. In particular, $|a_i^H|=|a_j^H|$ for all $1\leq i,j\leq n$. Certainly $\overline{H}= 1\,\cup\, \overline{a_1}^{\overline{H}}\,\cup\, \cdots \,\cup\, \overline{a_n}^{\overline{H}}$. Since $\overline{H}$ has prime power order, then it has non-trivial centre, and there must exist $1\leq i\leq n$ such that $|\overline{a_i}^{\overline{H}}|=1$. Therefore, $|\overline{a_j}^{\overline{H}}|=1$ for every $1\leq j \leq n$, because these conjugacy classes are conjugated in $\overline{G}$. Hence $\overline{H}$ is abelian. Moreover, some $\overline{a_i}$ should have prime order and, by the same previous argument, all elements of $\overline{H}$ have prime order. We conclude that $\overline{H}$ is elementary abelian.

We next prove that $\pi(a^G)=\pi(G)$ when $H\cap\ze{G}>1$. Observe that $(az)^G=a^Gz=a^G$ for any $z\in H\cap\ze{G}$, so $o(z)$ divides $|a^G|$. If $z$ is non-trivial, then $|a^G|$ cannot be odd; otherwise $a$ is an involution because $a^G$ is real, so $H$ is a $2$-group and thus $z$ is also a $2$-element, which is not possible. Therefore $|a^G|$ is even whenever $H\cap\ze{G}>1$. In fact, in such case if $|a^G|$ is a power of $2$, then every $2$-element of $G$ has class size a $2$-power, so $G=H\times \rad{2'}{G}$ by Proposition \ref{clave} (a); but this cannot occur since $\Gamma(H)=\Gamma_p(G)$ would be formed by a unique vertex.

Now if $H$ is abelian, then $\pi(a^G)=\{p\}$, so $H\cap\ze{G}=1$ by the previous paragraph and $k_{p'}(G)=2$. Thus the cases described in statements (a-f) follow from \cite[Theorem A]{N1}. Conversely, in all cases (a-f), the $p$-complements of $G/\rad{p}{G}$, and so those of $G$, are abelian.
\end{proof}

\medskip

\begin{remark}
In all cases (a-f) of the previous result it holds $k_{p'}(G/\rad{p}{G})=k_{p'}(G)=2$. One can easily find examples of groups $G$ satisfying the hypotheses of Proposition \ref{unique_vertex} with $k_{p'}(G)=3$, and thus the structure of the $p$-complements of $G$ is also known due to \cite{N1,N2}. Nevertheless, there are groups $G$ satisfying the hypotheses of Proposition \ref{unique_vertex} with $k_{p'}(G)=4$ and $p=2$, so it is not possible to apply the classification given in \cite{T}. For example, take $H$ an extraspecial group of order $27$ and exponent $3$, and let $S$ be a Sylow $2$-subgroup of $\op{Aut}(H)$. Consider the unique quaternion subgroup $Q$ of $S$, and let $G=H\rtimes Q$. Then $\Gamma_2(G)$ is formed by a unique $p$-regular class (whose length is $24$), and $|H\cap\ze{G}|=3$.
\end{remark}

\smallskip

\begin{proposition}
\label{two-connected}
Let $G$ be a $p$-separable group for a fixed prime $p$, and let $H$ be a $p$-complement of $G$. If $\Gamma_p(G)$ has two vertices and one edge, then one of the following cases holds:
\vspace*{-2mm}
\begin{enumerate}[label=\emph{(\alph*)}]
\setlength{\itemsep}{-1mm}
\item $H$ is a $q$-group, for some prime $q\neq p$.
\item $H$ is a Frobenius group with $q$-elementary abelian kernel and complements of order $r$, for two different primes $q$ and $r$ both distinct from $p$.
\end{enumerate}
\end{proposition}

\begin{proof}
Let $G$ be a counterexample of minimal order. We first suppose that $H/(H\cap\ze{G})$ is a $q$-group, for some prime $q$. In view of statement (a), $H\cap\ze{G}$ cannot be a $q$-group, so we may suppose that $H\cap\ze{G}$ is non-trivial. Lemma \ref{lemma_centre} yields $|H\cap\ze{G}|=2$. Thus $q$ is odd, and $H$ is a nilpotent $\{2,q\}$-group because $H/(H\cap\ze{G})$ is a $q$-group. We can then write $H=Q\times R$ where $R=\langle y\rangle \leqslant \ze{G}$ is cyclic of order $2$, and $Q$ is a $q$-group. Note that the two vertices of $\Gamma_p(G)$ are exactly $x^G$ and $(xy)^G$, where $x\in Q\smallsetminus \ze{G}$; moreover, both conjugacy classes are real because their representatives have different orders. Since $R$ is a central Sylow $2$-subgroup of $G$, then both class sizes are odd, so $x$ and $xy$ must be involutions which is not possible.

We suppose now that $H/(H\cap\ze{G})$ has not prime power order, so there exist two elements $x,y\in H\smallsetminus (H\cap \ze{G})$ such that $x$ is a $q$-element and $y$ is a $r$-element, for certain primes $q\neq r$. Due to the structure of $\Gamma_p(G)$, any other element of $H\smallsetminus (H\cap\ze{G})$ necessarily has order a $q$-number or a $r$-number, and $H\cap \ze{G}=1$. It follows that $H$ is a $\{q,r\}$-group with no elements of order divisible by $q$ and $r$ simultaneously. Note that $H$ is soluble, and since $G$ is $p$-separable by assumptions, then $G$ is also soluble.

Let $M$ be a minimal normal subgroup of $G$, which is either a $p$-group or a $p'$-group, and set $\overline{G}=G/M$. Let us first suppose that $M$ is a $p$-group. We are going to analyse three different possibilities for the graph $\Gamma_p(\overline{G})$, which certainly has at most two vertices, although we do not know a priori whether they are adjacent. If $\overline{x}^{\overline{G}}$ and $\overline{y}^{\overline{G}}$ are both non-central and adjacent in $\Gamma_p(\overline{G})$, then by minimality we necessarily get that $\overline{H}$ has the structure described in (a) or (b). But this is not possible because $\overline{H}$ is isomorphic to $H$. In addition, Proposition \ref{two-disconnected} yields that $\Gamma_p(\overline{G})$ cannot have two vertices and no edges. Hence either $\overline{x}^{\overline{G}}$ or $\overline{y}^{\overline{G}}$ must be central, so either $[x,H]\leqslant H\cap M=1$ or $[y,H]=1$, which cannot occur because $H$ has no element of order divisible simultaneously by $q$ and $r$. 

Finally we may suppose, up to interchanging $q$ and $r$, that $M$ is a $q$-elementary abelian group. As $M\nleqslant \ze{G}$ and $M$ is normal in $G$, it follows that $M=1\cup x^G\in\syl{q}{G}$ because $x^G$ contains all the $q$-elements of $H$. So $H=MR$ with $R\in\syl{r}{H}$, and all non-trivial elements of $R$ have order $r$. The fact that $H$ has no element of order divisible simultaneously by $q$ and $r$ implies that the action of $R$ on $M$ is Frobenius, and then $R$ must be cyclic of order $r$ because it is a Frobenius complement with all elements of order $r$. This is not possible due to statement (b), and the result is now proved.
\end{proof}

\medskip

\begin{remark}
\label{remark_many_classes}
In Proposition \ref{two-connected} (a) the group $G$ may have more than four $p$-regular classes and so it is difficult to detail further the structure of the $p$-complements of $G$. For instance, take $p=3$ and $G=C_2\times (Q_8\rtimes C_9)$, so $\Gamma_p(G)$ consists in two vertices with cardinalities $6$, but $H\cap\ze{G}=C_2\times C_2$ and thus there are $6$ $p$-regular classes.

However, in Proposition \ref{two-connected} (b) we can obtain the structure of $G/\rad{p}{G}$ by applying \cite{N1,N2}, because $G$ possesses exactly three $p$-regular classes, and therefore $G/\rad{p}{G}$ also has that feature. Note that the unique condition required to $G/\rad{p}{G}$ is to have at least three prime divisors, since $\Gamma_p(G/\rad{p}{G})$ may be connected or not, and even it may have less than two vertices. Hence cases (2), (11) and (12) of \cite[Theorem B]{N1}, and all cases in \cite[Theorem]{N2} may occur.
\end{remark}

\medskip

\begin{proposition}
\label{three-connected}
Let $G$ be a $p$-separable group for a fixed prime $p$, and let $H$ be a $p$-complement of $G$. If $\Gamma_p(G)$ is a tree of three vertices, then $p=3$ and $$G/\rad{p}{G}\cong (C_5 \times C_5) \rtimes SL(2,3),$$ where $SL(2,3)$ acts transitively on the non-trivial $5$-elements.
\end{proposition}

\begin{proof}
Let $x^G$ and $y^G$ be the two vertices of $\Gamma_p(G)$ such that $(|x^G|, |y^G|)=1$, and let $z^G$ be the vertex that is adjacent with both $x^G$ and $y^G$ in $\Gamma_p(G)$. If $a\in H\cap \ze{G}$, then $(ax)^G=ax^G=x^G$ since $x^G$ is the unique $p$-regular conjugacy class of $G$ with size $|x^G|$. It follows that $|\langle a\rangle|$ divides $|x^G|$, and analogously $|\langle a\rangle|$ divides $|y^G|$ too. Thus $H\cap\ze{G}=1$, and consequently $G$ has four $p$-regular classes. Moreover, both $x^G$ and $y^G$ are real classes, and some of them has odd length, so we may affirm that $2\in\pi(H)$ and $p$ is odd. 

If $H$ has prime power order, then $H$ is clearly a $2$-group. We claim that both $x$ and $y$ are involutions, and that the exponent of $H$ is $4$. First we show that both $x$ and $y$ are involutions. Otherwise we may suppose that $o(y)\geq 4$. Since $y$ and
$y^2$ are non-central elements of different orders, then $y^G\in\{x^G,z^G\}$ and we get a
contradiction as $|(y^2)^G|$ divides $|y^G|$. Note that this argument is also valid for $x$. Observe that $z$ cannot have order strictly larger than $4$, since otherwise we can form a triangle in $\Gamma_p(G)$. Besides, if $z$ has order two, then all elements of $H$ are involutions and $H$ would be abelian, which is not possible since it would imply that every $p$-regular class has $p$-power size and $\Gamma_p(G)$ would be a complete graph, a contradiction. 

If $H$ has not prime power order, then we claim that $|\pi(H)|=2$. Working by contradiction and recalling that $H\cap\ze{G}=1$, we may assume that $\pi(H)=\{q,r,s\}$ for pairwise different primes, all distinct from $p$, and that $x$ is a $q$-element, $z$ is a $r$-element, and $y$ is a $s$-element. It is easy to see that one of them, say $x$, necessarily has class size not divisible by some prime in $\pi(H)\smallsetminus\pi(o(x))$, say $r$, since otherwise we get a triangle in $\Gamma_p(G)$. It follows that there exists a non-trivial $\{q,r\}$-element $w\in G$, but $w^G$ is non-central and clearly different from $\{x^G,y^G,z^G\}$, a contradiction.

From the above paragraph we deduce that $H$ is soluble. Recall that $p$ must be odd. Certainly $\overline{G}=G/\rad{p}{G}$ has four $p$-regular conjugacy classes, so it remains to check which cases of \cite[Main Theorem]{T} may occur. We distinguish two cases: either $H$ has prime power order or it has not. In the former case, it is enough to look for the cases of \cite[Main Theorem]{T} such that $H$ is a $2$-group with exponent $4$. Case A) is discarded because $\overline{G}$ is isomorphic to $\Sigma_4$ with $p=3$, so $\Gamma_p(\overline{G})$ is a triangle and  by class length divisibility we would reach a contradiction. From the remaining cases of \cite[Main Theorem]{T} the unique possibility is B) 3. Nevertheless, in this last situation, $H$ is a Suzuki $2$-group of type A, so all the involutions belong to $\ze{H}$ and indeed they are conjugated, which certainly cannot occur. 

If $H$ has not prime power order, then the possibilities for $\overline{G}$ are B) 1 (c), 1 (d), 4, 5 and 6. In B) 1 (c), $\overline{G}$ has order not divisible by $p$, and $\overline{G}\cong H$ is an alternating group of degree $4$ with $p\neq 3$, so the non-central ($p$-regular) conjugacy classes are $\{(1,2)(3,4)^{H},(1,2,3)^{H},(1,2,4)^{H}\}$, whose sizes are $3$, $4$ and $4$ respectively. Observe that $(1,2,3)^{H}$ is the inverse class of $(1,2,4)^{H}$, and $H\cong \overline{G}$, so the corresponding class lengths in $G$ are also equal. But this cannot occur since $\Gamma_p(G)$ has no triangles. 

In B) 1 (d) we have that $G=\rad{p}{G}\rtimes H$ where $H\cong \overline{G}$ is isomorphic to a dihedral group of order $10$. Therefore the non-trivial ($p$-regular) classes of $\overline{G}$ are $\{\overline{u}^{\overline{G}},\overline{v}^{\overline{G}},\overline{w}^{\overline{G}}\}$ whose sizes are $5$, $2$ and $2$ respectively. In particular $\pi(u^G) \subseteq \{5,p\}$ and $\pi(v^G)\cup\pi(w^G)\subseteq \{2,p\}$. Consequently, since $\Gamma_p(G)$ is a path of length two, it must hold that $|u^G|$ and one of $\{|v^G|, |w^G|\}$, say $|v^G|$, are divisible by $p$. Since we may suppose up to conjugation that $\langle v\rangle=\langle w\rangle\in\syl{5}{G}$, then $p$ also divides $|w^G|$, so $\Gamma_p(G)$ is a triangle, a contradiction.

Let us consider now the case B) 6, so $p=5$ and $\overline{G}$ is isomorphic to $(C_2\times C_2\times C_2\times C_2)\rtimes C_{15}$, where $C_{15}$ permutes the involutions transitively. Hence the $p$-regular classes are $\{\overline{u}^{\overline{G}},\overline{v}^{\overline{G}},\overline{w}^{\overline{G}}\}$ whose sizes are $15$, $16$ and $16$ respectively. In fact $\overline{v}^{\overline{G}}$ and $\overline{w}^{\overline{G}}$ are inverse classes, so $|v^G|=|w^G|$, which is a contradiction since the graph $\Gamma_p(G)$ is triangle-free.

Finally, suppose that $\overline{G}$ has the structure described in B) 5, so it is isomorphic to $N\rtimes (P\times C_2)$, where $P$ is a cyclic $p$-group, $N$ is $q$-group, and $P\times C_2$ has two orbits on $N\smallsetminus \{1\}$, so $|N|=q^a$ and $q^a-1=4p^c$ for certain positive integers $a$ and $c$. Observe that $C_2$ inverts all elements of $N\smallsetminus \{1\}$, which yields that so $N$ is an elementary abelian $q$-group. Hence the $p$-regular classes of $\overline{G}$ are $\{\overline{u}^{\overline{G}},\overline{v}^{\overline{G}},\overline{w}^{\overline{G}}\}$ whose sizes are $2p^c$, $2p^c$ and $q^a$ respectively. By divisibility $\{2,p\}\subseteq \pi(u^G)=\pi(v^G)$, and equality holds because the Sylow $q$-subgroups are abelian. Now it follows that $\Gamma_p(G)$ is either non-connected or a triangle, a contradiction. This proves that $\overline{G}$ must have the structure of B) 4, as claimed.
\end{proof}

\section{Proof of Theorem \ref{teoA}}
\label{sec-A}

We notice that, in all cases described in the previous two sections, $G$ is indeed soluble. The result below shows that this fact holds when $\Gamma_p(G)$ has no triangles and $G$ is $p$-separable. We point out that the classification of finite simple groups is used, since the proof is based on \cite[Proposition 5]{BHM}.

\begin{proposition}
\label{soluble-prop}
If $G$ is a $p$-separable group such that $\Gamma_p(G)$ has no triangles, for some prime $p$, then $G$ is soluble.
\end{proposition}

\begin{proof}
It is enough to show that any chief factor $N/M$ of $G$ with order not divisible by $p$ is soluble. By contradiction, suppose that $N/M$ is non-soluble. Hence it is isomorphic to the direct product $N_1/M\times \cdots \times N_t/M$, where all direct factors are isomorphic copies of a simple non-abelian group. By \cite[Proposition 5]{BHM} we know that $\Gamma(N_1/M)$ is complete, so we can build a triangle in such graph formed by three classes whose representatives have prime order, for three pairwise distinct primes. Since these elements are $p$-regular, and $N_1\unlhd N \unlhd G$, we may therefore deduce that $\Gamma_p(G)$ has a triangle, against our assumptions.
\end{proof}

The result below is a particular case of Theorem \ref{teoA}, namely when every $p$-regular element has prime power order.

\begin{proposition}
\label{CP2}
Let $G$ be a $p$-separable group, for some prime $p$, and let $H$ be a $p$-complement of $G$. If $\Gamma_p(G)$ has no triangles, and every element of $H$ has prime power order, then either $H$ is a $q$-group, or it is a Frobenius $\{q,r\}$-group for two different primes $q$ and $r$ both distinct from $p$.
\end{proposition}

\begin{proof}
Suppose that $|H|$ is not a power of a prime. Since there is no element in $H$ with order divisible by two different primes, it follows $\ze{H}=1=H\cap\ze{G}$. Now Proposition \ref{soluble-prop} leads to the solubility of $G$, and so $H$ has the structure described in Theorem \ref{CP} (a), (b) or (c). Thus $|\pi(H)|=2$, and in the first two cases $H$ is a Frobenius group, so it is enough to show that statement (c) of that result cannot occur. 

Arguing by contradiction, let us suppose that $1\neq \rad{q}{H}$ for some prime $q$ and $H/\rad{q}{H}$ is a Frobenius group with (cyclic) kernel $K/\rad{q}{H}$. Certainly $K=R\rad{q}{H}$ where $R\in\syl{r}{H}$, for some prime $r\neq q$, and $K$ is a Frobenius group because every element of $H$ has prime power order. Moreover, $r$ divides $|b^G|$ for all $b\in \rad{q}{H}$, since otherwise there exists a $r$-element that commutes with $b$, which is not possible because the $p$-regular elements of $G$ all have prime power order. Using similar arguments, observe that $r$ also divides $|z^G|$ for any $q$-element $z\in H\smallsetminus \rad{q}{H}$, and $q$ divides $|a^G|$ for every element $a\in R$. 

We claim that the unique vertices of $\Gamma_p(G)$ are $\{a^G, b^G, z^G\}$. In fact, if there exists another vertex $c^G$ of the graph $\Gamma_p(G)$, then $c$ is either a $q$-element or a $r$-element. The former case, analogously as before, yields that $r$ divides $|c^G|$, so the classes $\{c^G, b^G, z^G\}$ form a triangle in $\Gamma_p(G)$, a contradiction. Hence we may suppose that $c$ is a $r$-element, and $q$ divides $|c^G|$. Since $\Gamma_p(G)$ has no triangles, we deduce that $q$ does not divide the class size in $G$ of any $q$-element of $H$, and therefore the same holds for any $q$-element of $G$. By Proposition \ref{clave} (b) we get that a Sylow $q$-subgroup of $H$ is abelian. Since every element of $H$ has prime power order, and $\rad{q}{H}\neq 1$, then $\rad{r}{H}=1$ and $\fit{H}=\rad{q}{H}$, so $z\in\ce{H}{\rad{q}{H}}\leqslant\rad{q}{H}$, a contradiction.

Thus the unique vertices of $\Gamma_p(G)$ are $\{a^G, b^G, z^G\}$. As the graph has no triangles and it has at most two connected components, then we are in situation to apply either Proposition \ref{three-disconnected} or Proposition \ref{three-connected}, and both results yield the final contradiction.
\end{proof}

\medskip

We are now ready to prove Theorem \ref{teoA} of the Introduction.

\medskip

\begin{proof}[Proof of Theorem \ref{teoA}]
Let us suppose that $H$ is a non-central $p$-complement of $G$, and that $\Gamma_p(G)$ has no triangles. Recall that $G$ is soluble by Proposition \ref{soluble-prop}. Certainly, if $\Gamma_p(G)$ is non-connected, then the claims follow by Corollary \ref{cor_disconnected_2}. As a consequence we may suppose henceforth that $\Gamma_p(G)$ is connected. Moreover, if $\Gamma_p(G)$ has at most three vertices, then we are done by Propositions \ref{unique_vertex}, \ref{two-connected} and \ref{three-connected}. Consequently it is enough to show that $\Gamma_p(G)$ cannot have more than three vertices. Arguing by contradiction, let us suppose that it possesses at least four vertices. We distinguish two cases.

\medskip

\noindent \textbf{\underline{Case I.}} There is some element in $H/(H\cap\ze{G})$ whose order is not a prime power.

\medskip

We claim that this case is not possible. Take $x(H\cap\ze{G})$ in $H/(H\cap\ze{G})$ of order divisible by exactly two different primes $q$ and $r$. Let us consider the decomposition $x=x_qx_r$ into its $q$-part and its $r$-part, being both clearly non-central in $G$. It follows that the class sizes in $G$ of $x_q$ and $x_r$ divide the class size of $x$, and these three classes are certainly different, so $\Gamma_p(G)$ has a subgraph of the form $$x_q^G\mbox{ --- }x^G\mbox{ --- }x_r^G.$$ Since there are no triangles in the graph, then $|x_q^G|$ and $|x_r^G|$ are coprime, and $G=\ce{G}{x_q}\ce{G}{x_r}$. It follows that $|x^G|=|x_q^G|\cdot|x_r^G|$, and thus any different vertex of $\Gamma_p(G)$ yields a triangle, which cannot occur. 

\medskip

\noindent \textbf{\underline{Case II.}} Every element of $H/(H\cap\ze{G})$ has prime power order.

\medskip

Let us first suppose that $H\cap\ze{G}>1$. We may assume that $\Gamma_p(G)$ contains a subgraph of the form $A\mbox{ --- }B\mbox{ --- }C$, where $A$ and $C$ are certainly the unique $p$-regular classes with cardinalities $|A|$ and $|C|$, respectively, because the graph has no triangles. Take a non-trivial $z\in H\cap\ze{G}$. Hence $zA=A$ and $zC=C$, so we deduce that $|\langle z\rangle|$ divides both $|A|$ and $|C|$, which is not possible. 

Now we suppose $H\cap \ze{G}=1$, so $|H|$ is divisible by at most two primes $q\neq r$ by Theorem \ref{CP}. Proposition \ref{CP2} yields that either $|H|$ is a prime power, or $H$ is a Frobenius $\{q,r\}$-group. Since $\pi(G)\subseteq \{p,q,r\}$, and $\Gamma_p(G)$ is connected, triangle-free, and its diameter at most $3$ (\emph{cf.} \cite[Theorem 2]{BF1}), then we cannot have more than four vertices: otherwise there must exist a vertex with size divisible by $pqr$, which would form a triangle with two of the remaining vertices in $\Gamma_p(G)$. We next analyse the two possible trees of four vertices and we will show that none of them can occur. Observe that the structure of such trees forces $\pi(H)=\{q,r\}$, so $|H|$ cannot be a prime power.

If $\Gamma_p(G)$ is a tree with the following structure \begin{eqnarray*} b^G\mbox{ --- } & a^G & \mbox{ --- }c^G \\
& \rotatebox{90}{---} & \\
& d^G & \end{eqnarray*} then $\pi(a^G)=\{p,q,r\}$, and we may assume that $\pi(d^G)=\{p\}$. It follows that $|d^G|$ is not divisible by either $q$ or $r$, which is not possible, because $H$ does not contain any element of order divisible by two different primes. 

Finally, if $\Gamma_p(G)$ has the structure $$a^G\mbox{ --- }b^G\mbox{ --- }c^G\mbox{ --- }d^G,$$ since there are no $p$-regular classes with size a $p$-power by a similar previous argument, then it necessarily follows $\pi(a^G)=\{q\}$, $\pi(b^G)=\{p,q\}$, $\pi(c^G)=\{p,r\}$ and $\pi(d^G)=\{r\}$. Observe that $a$ and $b$ are $r$-elements, and $c$ and $d$ are $q$-elements, because $H$ does not have elements of orders divisible by two primes. We deduce by Proposition \ref{clave} (b) that $G$ has abelian Sylow $q$-subgroups and Sylow $r$-subgroups. Let $Q\in\syl{q}{H}$ and $R\in\syl{r}{H}$, which are both abelian, and we may assume $a,b\in R$ and $c,d\in Q$. We have seen above that $H$ is a Frobenius $\{q,r\}$-group, so $|a^G|=|Q|$, $|b^G|=|Q|\cdot p^{\alpha}$, $|c^G|=|R|\cdot p^{\beta}$ and $|d^G|=|R|$, for certain integers $\alpha,\beta>0$. The number of $r$-elements of $G$ is therefore $1+|Q|+|Q|\cdot p^{\alpha}=1+|Q|\cdot (1+p^{\alpha})$. Analogously, the number of $q$-elements of $G$ is $1+|R|\cdot (1+p^{\beta})$. Note that all $p$-regular conjugacy classes of $G$ are real, and in particular some of them has odd length, which implies that $2\in\pi(H)$ and so $p$ is odd. Hence the number of $q$-elements of $G$, and the number of $r$-elements of $G$, are both odd. But these numbers must be multiple of $|R|$ and $|Q|$, respectively, by a celebrated theorem due to Frobenius (\emph{cf.} \cite[Theorem 9.1.2]{Hall}). Hence we have obtained a contradiction to the fact $2\in\pi(H)$. This final contradiction proves the theorem.
\end{proof}


\bigskip

\noindent\textbf{Acknowledgements.} The authors would like to thank H. Tong-Viet and J. Cossey  for helpful discussions about Proposition \ref{three-connected}, and the referees for their valuable comments and suggestions.


\end{document}